\documentclass[10pt]{amsart}
\usepackage{amssymb}
\usepackage{bm}
\usepackage{graphicx}
\usepackage[centertags]{amsmath}
\usepackage{amsfonts}
\usepackage{amsthm}
\newtheorem{thm}{Theorem}
\newtheorem{cor}[thm]{Corollary}
\newtheorem{lem}[thm]{Lemma}

\newtheorem{prop}[thm]{Proposition}

\newtheorem{fact}[thm]{Fact}
\newtheorem{defn}[thm]{Definition}
\theoremstyle{definition}

\newtheorem{examp}{Example}

\newcommand{\rr}{\mathbb{R}}
\newcommand{\nn}{\mathbb{N}}
\newcommand{\ee}{\varepsilon}

\newcommand{\sbs}{\mathrm{SB}}
\newcommand{\us}{\mathrm{US}}
\newcommand{\txtwo}{T^{\mathfrak{X}}_2}
\newcommand{\txco}{T^{\mathfrak{X}}_0}

\newcommand{\stblng}{\mathfrak{X}=(X,\Lambda,T,(x_t)_{t\in T})}
\newcommand{\sg}{\sigma}
\newcommand{\seg}{\mathfrak{s}}

\newcommand{\ospan}{\overline{\mathrm{span}}}
\newcommand{\llll}{\mathcal{L}}
\newcommand{\ccc}{\mathcal{C}}
\newcommand{\xxx}{\mathcal{X}}
\newcommand{\supp}{\mathrm{supp}}
\newcommand{\range}{\mathrm{range}}
\newcommand{\aaa}{\mathcal{A}}

\begin{document}

\title{On unconditionally saturated Banach spaces}
\author{Pandelis Dodos and Jordi Lopez-Abad}

\address{Universit\'{e} Pierre et Marie Curie - Paris 6, Equipe d' Analyse
Fonctionnelle, Bo\^{i}te 186, 4 place Jussieu, 75252 Paris Cedex 05, France.}
\email{pdodos@math.ntua.gr}

\address{Universit\'{e} Denis Diderot - Paris 7, Equipe de Logique
Math\'{e}matiques, 2 place Jussieu, 72521 Paris Cedex 05, France.}
\email{abad@logique.jussieu.fr}


\maketitle


\begin{abstract}
We prove a structural property of the class of unconditionally
saturated separable Banach spaces. We show, in particular,
that for every analytic set $\aaa$, in the Effros-Borel space of
subspaces of $C[0,1]$, of unconditionally saturated separable
Banach spaces, there exists an unconditionally saturated Banach
space $Y$, with a Schauder basis, that contains isomorphic copies
of every space $X$ in the class $\aaa$.
\end{abstract}


\section{Introduction}

\noindent \textbf{(A)} An infinite-dimensional Banach space $X$ is said
to be \textit{unconditionally saturated} if every infinite-dimensional
subspace $Y$ of $X$ contains an unconditional basic sequence. Although
by the discovery of W. T. Gowers and B. Maurey \cite{GM} not every
separable Banach space is unconditionally saturated, this class of
spaces is quite extensive, includes the ``classical" ones and has
some desirable closure properties (it is closed, for instance, under
taking subspaces and finite sums). Most important is the fact that
within the class of unconditionally saturated spaces one can develop
a strong structural theory. Among the numerous results found in the
literature, there are two fundamental ones that deserve special
attention. The first is due to R. C. James \cite{Ja1} and asserts
that any unconditionally saturated space contains either a reflexive
subspace, or $\ell_1$, or $c_0$. The second is due to A. Pe{\l}czy\'{n}ski
\cite{P} and provides a space $U$ with an unconditional basis $(u_n)$
with the property that any other unconditional basic sequence
$(x_n)$, in some Banach space $X$, is equivalent to a subsequence of $(u_n)$.
\medskip

\noindent \textbf{(B)} The main goal of this paper is to exhibit
yet another structural property of the class of unconditionally saturated
spaces which is of a global nature. To describe this property we need
first to recall some standard facts. Quite often one needs a convenient
way to treat separable Banach spaces as a unity. Such a way has been
proposed by B. Bossard \cite{Bos} and has been proved to be extremely
useful. More precisely, let us denote by $F\big(C[0,1]\big)$ the set
of all closed subspaces of the space $C[0,1]$ and let us consider the set
\begin{equation}
\label{e1} \sbs=\big\{X\in F\big(C[0,1]\big): X \text{ is a linear subspace}\big\}.
\end{equation}
It is easy to see that the set $\sbs$ equipped with the
relative Effros-Borel structure becomes a standard Borel space
(see \cite{Bos} for more details). As $C[0,1]$ is isometrically
universal for all separable Banach spaces, we may identify any
class of separable Banach spaces with a subset of $\sbs$.
Under this point of view, we denote by $\us$ the subset of
$\sbs$ consisting of all $X\in\sbs$ which are unconditionally
saturated.

The above identification is ultimately related to universality
problems in Banach Space Theory (see \cite{AD}, \cite{DF}, \cite{D}).
The connection is crystalized in the following definition, introduced
in \cite{AD}.
\begin{defn}
\label{ind1} A class $\ccc\subseteq \sbs$ is said to be strongly
bounded if for every analytic subset $\aaa$ of $\mathcal{C}$ there
exists $Y\in\mathcal{C}$ that contains isomorphic copies of every
$X\in \aaa$.
\end{defn}

In \cite[Theorem 91(5)]{AD} it was shown that the class of unconditionally
saturated Banach spaces with a Schauder basis is strongly bounded.
We remove the assumption of the existence of a basis and we show
the following.
\begin{thm}
\label{int1} Let $\aaa$ be an analytic subset of $\us$. Then
there exists an unconditionally saturated Banach space $Y$,
with a Schauder basis, that contains isomorphic copies of every
$X\in\aaa$.

In particular, the class $\us$ is strongly bounded.
\end{thm}
We should point out that the above result is optimal. Indeed,
it follows by a classical construction of J. Bourgain
\cite{Bou1} that there exists a co-analytic subset $\mathcal{B}$
of $\sbs$ consisting of reflexive and unconditionally
saturated separable Banach spaces with the following property.
If $Y$ is a separable space that contains an isomorphic copy
of every $X\in\mathcal{B}$, then $Y$ must contain every
separable Banach space. In particular, there is no
unconditionally saturated separable Banach space containing
isomorphic copies of every $X\in\mathcal{B}$.
\medskip

\noindent \textbf{(C)} By the results in \cite{AD}, the proof of
Theorem \ref{int1} is essentially reduced to an embedding problem.
Namely, given an unconditionally saturated separable Banach space
$X$ one is looking for an unconditionally saturated space $Y(X)$,
with a Schauder basis, that contains an isomorphic copy of $X$.
In fact, for the proof of Theorem \ref{int1}, one has to know
additionally that this embedding is ``uniform". This means,
roughly, that the space $Y(X)$ is constructed from
$X$ in a Borel way. In our case, the embedding problem has been
already solved by J. Bourgain and G. Pisier in \cite{BP}, while
its uniform version has been recently obtained in \cite{D}.
These are the main ingredients of the proof of Theorem \ref{int1}.
\medskip

\noindent \textbf{(D)} At a more technical level, the paper also
contains some results concerning the structure of a class of subspaces
of a certain space constructed in \cite{AD} and called as an $\ell_2$
Baire sum. Specifically, we study the class of $X$-singular subspaces
of an $\ell_2$ Baire sum and we show the following (see \S 3.1 for the
relevant definitions).
\medskip

(1) Every $X$-singular subspace is unconditionally
saturated (Theorem \ref{t37} in the main text).
\medskip

(2) Every $X$-singular subspace contains an $X$-compact
subspace (Corollary \ref{c312} in the main text). This answers a question
from \cite{AD} (see \cite[Remark 3]{AD}).
\medskip

(3) Every normalized basic sequence in an $X$-singular
subspace has a normalized block subsequence satisfying an upper
$\ell_2$ estimate (Theorem \ref{t38} in the main text). Hence, an
$X$-singular subspace can contain no $\ell_p$ for $1\leq p<2$.
This generalizes the fact that the 2-stopping time Banach space
(see \cite{BO}) can contain no $\ell_p$ for $1\leq p<2$.

\subsection{General notation and terminology} By $\nn=\{0,1,2,...\}$
we shall denote the natural numbers. For every infinite subset
$L$ of $\nn$, by $[L]$ we denote the set of all infinite
subsets of $L$. Our Banach space theoretic notation and
terminology is standard and follows \cite{LT}, while
our descriptive set theoretic terminology follows \cite{Kechris}.
If $X$ and $Y$ are Banach spaces, then we shall denote
the fact that $X$ and $Y$ are isomorphic by $X\cong Y$.

For the convenience of the reader, let us recall the following
notions. A measurable space $(X,S)$ is said to be a
\textit{standard Borel space} if there exists a Polish
topology\footnote[1]{A topology $\tau$ on a set $X$ is said
to be Polish if the space $(X,\tau)$ is a separable
and completely metrizable topological space.} $\tau$ on $X$
such that the Borel $\sigma$-algebra of $(X,\tau)$ coincides
with $S$. A subset $B$ of a standard Borel space $(X,S)$ is
said to be \textit{analytic} if there exists a Borel map
$f:\nn^\nn\to X$ such that $f(\nn^\nn)=B$. Finally,
a seminormalized sequence $(x_n)$ in a
Banach space $X$ is said to be \textit{unconditional}
if there exists a constant $C>0$ such that for every $k\in\nn$,
every $F\subseteq \{0,...,k\}$ and every $a_0,...,a_k\in\rr$
we have
\begin{equation}
\label{e13} \big\| \sum_{n\in F} a_n x_n \big\| \leq C \|\sum_{n=0}^k a_n x_n\|.
\end{equation}

\subsection{Trees} The concept of a tree has been proved to
be a very fruitful tool in the Geometry of Banach spaces.
It is also decisive throughout this work. Below we gather all
the conventions concerning trees that we need.

Let $\Lambda$ be a non-empty set. By $\Lambda^{<\nn}$ we shall
denote the set of all \textit{non-empty}\footnote[2]{We should
point out that in many standard textbooks, as for instance in
\cite{Kechris}, the empty sequence is included in $\Lambda^{<\nn}$.
We do not include the empty sequence for technical reasons that
will become transparent in \S 3.} finite sequences in $\Lambda$.
By $\sqsubset$ we shall denote the (strict) partial order on $\Lambda^{<\nn}$
of end-extension. For every $\sg\in \Lambda^\nn$ and every $n\in\nn$
with $n\geq 1$ we set $\sg|n=\big(\sg(0),..., \sg(n-1)\big)\in\Lambda^{<\nn}$.
Two nodes $s,t\in\Lambda^{<\nn}$ are said to be \textit{comparable}
if either $s\sqsubseteq t$ or $t\sqsubseteq s$; otherwise they are
said to be \textit{incomparable}. A subset of $\Lambda^{<\nn}$
consisting of pairwise comparable nodes is said to be a
\textit{chain}, while a subset of $\Lambda^{<\nn}$ consisting
of pairwise incomparable nodes is said to be an \textit{antichain}.

A \textit{tree} $T$ on $\Lambda$ is a subset of $\Lambda^{<\nn}$
satisfying
\begin{equation}
\label{e2} \forall s,t\in\Lambda^{<\nn} \ (t\in T \text{ and }
s\sqsubset t\Rightarrow s\in T).
\end{equation}
A tree $T$ is said to be \textit{pruned} if for every
$s\in T$ there exists $t\in T$ with $s\sqsubset t$.
The \textit{body} $[T]$ of a tree $T$ on $\Lambda$ is
defined to be the set $\{\sg\in\Lambda^\nn: \sg|n\in T
\ \forall n\geq 1\}$. Notice that if $T$ is pruned, then
$[T]\neq\varnothing$. A \textit{segment} $\seg$ of a tree
$T$ is a chain of $T$ satisfying
\begin{equation}
\label{e3} \forall s,t,w\in \Lambda^{<\nn} \ (s\sqsubseteq w
\sqsubseteq t \ \text{ and } s,t\in\seg\Rightarrow w\in\seg).
\end{equation}
If $\seg$ is a segment of $T$, then by $\min(\seg)$ we denote
the $\sqsubseteq$-minimum node $t\in\seg$. We say that two
segments $\seg$ and $\seg'$ of $T$ are \textit{incomparable}
if for every $t\in\seg$ and every $t'\in\seg'$ the nodes $t$
and $t'$ are incomparable (notice that this is equivalent to
say that $\min(\seg)$ and $\min(\seg')$ are incomparable).


\section{Embedding unconditionally saturated spaces into spaces
with a basis}

The aim of this section is to give the proof of the following
result.
\begin{prop}
\label{p21} Let $\aaa$ be an analytic subset of $\us$. Then there
exists an analytic subset $\aaa'$ of $\us$ with the following
properties.
\begin{enumerate}
\item[(i)] For every $Y\in \aaa'$ the space $Y$ has a Schauder
basis\footnote[3]{Throughout the paper, when we say that a
Banach space $X$ has a Schauder basis, then we implicitly
assume that $X$ is infinite-dimensional.}.
\item[(ii)] For every $X\in \aaa$ there exists $Y\in \aaa'$
that contains an isometric copy of $X$.
\end{enumerate}
\end{prop}
As we have already mention in the introduction, the proof of Proposition
\ref{p21} is based on a construction of $\llll_\infty$-spaces due to
J. Bourgain and G. Pisier \cite{BP}, as well as, on its parameterized
version which has been recently obtained in \cite{D}.

Let us recall, first, some definitions. If $X$ and $Y$ are two isomorphic Banach
spaces (not necessarily infinite-dimensional), then their \textit{Banach-Mazur
distance} is defined by
\begin{equation}
\label{e4} d(X,Y)=\inf\big\{ \|T\|\cdot \|T^{-1}\|: T:X\to Y
\text{ is an isomorphism}\big\}.
\end{equation}
Let now $X$ be an infinite-dimensional Banach space and $\lambda\geq 1$.
The space $X$ is said to be a $\llll_{\infty,\lambda}$-space if for every
finite-dimensional subspace $F$ of $X$ there exists a finite-dimensional
subspace $G$ of $X$ with $F\subseteq G$ and $d(G,\ell^n_\infty)\leq\lambda$,
where $n=\mathrm{dim}(G)$. The space $X$ is said to be a
$\llll_{\infty,\lambda+}$-space if it is a $\llll_{\infty,\theta}$-space
for every $\theta>\lambda$. Finally, $X$ is said to be a
$\llll_{\infty}$-space if it is $\llll_{\infty,\lambda}$ for
some $\lambda\geq 1$. The class of $\llll_{\infty}$-spaces
was defined by J. Lindenstrauss and A. Pe{\l}czy\'{n}ski \cite{LP}.
For a comprehensive account of the theory of $\llll_\infty$-spaces,
as well as, for a presentation of many remarkable examples
we refer to the monograph of J. Bourgain \cite{Bou2}.

Let us also recall that a Banach space $X$ is said to have the
\textit{Schur} property if every weakly convergent sequence in $X$
is automatically norm convergent. It in an immediate consequence
of Rosenthal's Dichotomy \cite{Ro} that every space $X$ with the
Schur property is hereditarily $\ell_1$; that is, every subspace
$Y$ of $X$ has a further subspace isomorphic to $\ell_1$ (hence,
every space with the Schur property is unconditionally saturated).

The following theorem summarizes some of the basic properties of
the Bourgain-Pisier construction.
\begin{thm}[\cite{BP}, Theorem 2.1]
\label{t22} Let $\lambda>1$ and $X$ be a separable Banach space.
Then there exists a separable $\llll_{\infty,\lambda+}$-space,
denote by $\llll_\lambda[X]$, which contains $X$ isometrically
and is such that the quotient $\llll_\lambda[X]/X$ has the
Radon-Nikodym and the Schur properties.
\end{thm}
The parameterized version of Theorem \ref{t22} reads as follows.
\begin{thm}[\cite{D}, Theorem 16]
\label{t23} For every $\lambda>1$, the set $\llll_\lambda\subseteq
\sbs\times\sbs$ defined by
\[ (X,Y)\in\llll_\lambda \Leftrightarrow Y \text{ is isometric to }
\llll_\lambda[X]\]
is analytic.
\end{thm}
We will also need the following Ramsey-type lemma. Although it is
well-known, we sketch its proof for completeness.
\begin{lem}
\label{l24} Let $X$ be a Banach space and $Y$ be a closed subspace
of $X$. Then, for every subspace $Z$ of $X$ there exists a further
subspace $Z'$ of $Z$ such that $Z'$ is either isomorphic to a subspace
of $Y$, or isomorphic to a subspace of $X/Y$.

In particular, if $Y$ and $X/Y$ are both unconditionally saturated,
then so is $X$.
\end{lem}
\begin{proof}
Let $Q:X\to X/Y$ be the natural quotient map. Consider the following
(mutually exclusive) cases.
\medskip

\noindent \textsc{Case 1.} \textit{The operator $Q:Z\to X/Y$ is not
strictly singular.} This case, by definition, yields the existence
of a subspace $Z'$ of $Z$ such that $Q|_{Z'}$ is an isomorphic
embedding.
\medskip

\noindent \textsc{Case 2.} \textit{The operator $Q:Z\to X/Y$ is
strictly singular.} In this case our hypothesis implies that for every
subspace $Z'$ of $Z$ and every $\ee>0$ we may find a normalized vector
$z\in Z'$ such that $\|Q(z)\|\leq \ee$. Hence, for every subspace
$Z'$ of $Z$ and every $\ee>0$ there exist a normalized vector
$z\in Z'$ and a vector $y\in Y$ such that $\|z-y\|<\ee$. So,
we may construct a normalized Schauder basic sequence $(z_n)$
in $Z$ with basis constant $2$ and a sequence $(y_n)$ in $Y$
such that $\|z_n-y_n\|<1/8^n$ for every $n\in\nn$. It follows
that $(y_n)$ is equivalent to $(z_n)$ (see \cite{LT}). Setting
$Z'=\overline{\mathrm{span}}\{z_n:n\in\nn\}$, we see that $Z'$
is isomorphic to a subspace of $Y$. The proof is completed.
\end{proof}
We are ready to proceed to the proof of Proposition \ref{p21}.
\begin{proof}[Proof of Proposition \ref{p21}]
Let $\aaa$ be an analytic subset of $\us$. Let also $\llll_2$
be the subset of $\sbs\times\sbs$ obtained by applying Theorem
\ref{t23} for $\lambda=2$. We define $\aaa'\subseteq \sbs$ by the
rule
\[ Y\in \aaa' \Leftrightarrow \exists X \ \big[ X\in \aaa \text{ and }
(X,Y)\in\llll_2\big].\]
As both $\aaa$ and $\llll_2$ are analytic and the class of analytic
sets is closed under projections, we see that $\aaa'$ is analytic.
We claim that $\aaa'$ is the desired set. Indeed, notice that property
(ii) is an immediate consequence of Theorem \ref{t22}. To see (i),
let $Y\in \aaa'$ arbitrary. There exists $X\in \aaa$ such that
$Y$ is isometric to $\llll_2[X]$. By Theorem \ref{t22}, we know
that $\llll_2[X]/X$ is unconditionally saturated. Recalling that $X$
is also unconditionally saturated, by Lemma \ref{l24}, we see that
$Y\in \us$. Finally, our claim that $Y$ has a Schauder basis is an
immediate consequence of the fact that $Y$ is $\llll_\infty$ and of
a classical result due to W. B. Johnson, H. P. Rosenthal and M. Zippin
\cite{JRZ} asserting that every separable $\llll_\infty$-space has a
Schauder basis. The proof is completed.
\end{proof}


\section{Schauder tree bases and $\ell_2$ Baire sums}


\subsection{Definitions and statements of the main results}
Let us begin be recalling the following notion.
\begin{defn}[\cite{AD}, Definition 13]
\label{d31} Let $X$ be a Banach space, $\Lambda$ a countable set
and $T$ a pruned tree on $\Lambda$. Let also $(x_t)_{t\in T}$ be
a normalized sequence in $X$ indexed by the tree $T$. We say that
$\stblng$ is a Schauder tree basis if the following are satisfied.
\begin{enumerate}
\item[(a)] $X=\ospan\{x_t:t\in T\}$.
\item[(b)] For every $\sg\in [T]$ the sequence
$(x_{\sg|n})_{n\geq 1}$ is a (normalized) bi-monotone
Schauder basic sequence.
\end{enumerate}
\end{defn}
Let $\stblng$ be a Schauder tree basis. For every $\sg\in [T]$ we set
\begin{equation}
\label{e5} X_\sg=\ospan\{ x_{\sg|n}:n\geq 1\}.
\end{equation}
Notice that in Definition \ref{d31} we do not assume that the
subspace $X_\sg$ of $X$ is complemented. Notice also that if
$\sg, \tau\in [T]$ with $\sg\neq\tau$, then this does not necessarily
imply that $X_\sg\neq X_\tau$.
\begin{examp}
\label{ex32} Let $X=c_0$ and $(e_n)$ be the standard unit vector basis
of $c_0$. Let also $T=2^{<\nn}$ be the Cantor tree; i.e. $T$ is the set of
all non-empty finite sequences of $0$'s and $1$'s. For every $t\in T$,
denoting by $|t|$ the length of the finite sequence $t$, we define
$x_t=e_{|t|-1}$. It is easy to see that the family $(X,2,T,(x_t)_{t\in T})$
is a Schauder tree basis. Observe that for every $\sg\in [T]$ the
sequence $(x_{\sg|n})_{n\geq 1}$ is the standard basis of $c_0$.
Hence, the just defined Schauder tree basis has been obtained by
``spreading" along the branches of $2^{<\nn}$ the standard basis of
$c_0$.
\end{examp}
The notion of a Schauder tree basis serves as a technical vehicle
for the construction of a ``tree-like" Banach space in the spirit
of R. C. James \cite{Ja2}. This is the content of the following
definition.
\begin{defn}[\cite{AD}, \S 4.1]
\label{d33} Let $\stblng$ be a Schauder tree basis.
The $\ell_2$ Baire sum of $\mathfrak{X}$,
denoted by $\txtwo$, is defined to be the completion
of $c_{00}(T)$ equipped with the norm
\begin{equation}
\label{e6} \|z\|_{\txtwo}= \sup\Big\{ \Big( \sum_{j=0}^l
\big\| \sum_{t\in \seg_j} z(t) x_t\big\|^2_X \Big)^{1/2} \Big\}
\end{equation}
where the above supremum is taken over all finite
families $(\seg_j)_{j=0}^l$ of pairwise incomparable
segments of $T$.
\end{defn}
\begin{examp}
\label{ex34} Let $\mathfrak{X}$ be the Schauder tree basis
described in Example \ref{ex32} and consider the corresponding
$\ell_2$ Baire sum $\txtwo$. Notice that if $z\in \txtwo$,
then its norm is given by the formula
\[ \|z\|_{\txtwo}=\sup\Big\{ \Big(\sum_{j=0}^l
z(t_j)^2\Big)^{1/2}: (t_j)_{j=0}^l \text{ is an antichain of }
2^{<\nn}\Big\}.\]
This space has been defined by H. P. Rosenthal and it is known
in the literature as the \textit{$2$-stopping time} Banach
space (see \cite{BO}). It is usually denoted by $S_2$. A very
interesting fact concerning the structure of $S_2$
is that it contains almost isometric copies of $\ell_p$ for
every $2\leq p<\infty$. This is due to H. P. Rosenthal and
G. Schechtman (unpublished). On the other hand, the space
$S_2$ can contain no $\ell_p$ for $1\leq p<2$.
\end{examp}
Let $\stblng$ be a Schauder tree basis and consider the corresponding
$\ell_2$ Baire sum $\txtwo$ of $\mathfrak{X}$. Let $(e_t)_{t\in T}$
be the standard Hamel basis of $c_{00}(T)$. We fix a bijection
$h:T\to\nn$ such that for every pair $t,s\in T$ we have $h(t)<h(s)$
if $t\sqsubset s$. If $(e_{t_n})$ is the enumeration of $(e_t)_{t\in T}$
according to $h$, then it is easy to verify that the sequence
$(e_{t_n})$ defines a normalized bi-monotone Schauder basis of $\txtwo$.

For every $\sg\in [T]$ consider the subspace $\xxx_\sg$ of $\txtwo$
defined by
\begin{equation}
\label{e7} \xxx_\sg=\ospan\{ e_{\sg|n}:n\geq 1\}.
\end{equation}
It is easily seen that the space $\xxx_\sg$ is isometric to $X_\sg$
and, moreover, it is $1$-complemented in $\txtwo$ via the natural
projection $P_\sg:\txtwo\to \xxx_\sg$. More generally, for every
segment $\seg$ of $T$ we set $\xxx_\seg=\ospan\{e_t:t\in\seg\}$.
Again we see that $\xxx_\seg$ is isometric to the space
$\ospan\{x_t: t\in\seg\}$ and it is 1-complemented in $\txtwo$
via the natural projection $P_\seg:\txtwo\to \xxx_\seg$.

If $x$ is a vector in $\txtwo$, then by $\supp(x)$ we shall
denote its \textit{support}; i.e. the set $\{t\in T: x(t)\neq 0\}$.
The \textit{range} of $x$, denoted by $\range(x)$, is defined
to be the minimal interval $I$ of $\nn$ satisfying
$\supp(x)\subseteq \{t_n:n\in I\}$. We isolate, for future
use, the following consequence of the enumeration $h$ of $T$.
\begin{fact}
\label{f35} Let $\seg$ be a segment of $T$ and $I$ be an interval
of $\nn$. Consider the set $\seg'=\seg\cap \{t_n: n\in I\}$. Then
$\seg'$ is also a segment of $T$.
\end{fact}
Let now $Y$ be a subspace of $\txtwo$. Assume that there exist
a subspace $Y'$ of $Y$ and a $\sg\in [T]$ such that the
operator $P_\sg:Y'\to\xxx_\sg$ is an isomorphic embedding. In such
a case, the subspace $Y$ contains information about the Schauder
tree basis $\stblng$. On the other hand, there are subspaces of
$\txtwo$ which are ``orthogonal" to every $\xxx_\sg$. These subspaces
are naturally distinguished into two categories, as follows.
\begin{defn}[\cite{AD}, Definition 14]
\label{d36} Let $\stblng$ be a Schauder tree basis
and let $Y$ be a subspace of $\txtwo$.
\begin{enumerate}
\item[(a)] We say that $Y$ is $X$-singular if for every $\sg\in [T]$
the operator $P_\sg:Y\to\xxx_\sg$ is strictly singular.
\item[(b)] We say that $Y$ is $X$-compact if for every $\sg\in [T]$
the operator $P_\sg:Y\to\xxx_\sg$ is compact.
\end{enumerate}
\end{defn}
In this section, we are focussed on the structure of the class
of $X$-singular subspaces of an arbitrary $\ell_2$ Baire sum.
Our main results are summarized below.
\begin{thm}
\label{t37} Let $\stblng$ be a Schauder tree basis and $Y$
be an $X$-singular subspace of $\txtwo$. Then $Y$ is unconditionally
saturated.
\end{thm}
\begin{thm}
\label{t38} Let $\stblng$ be a Schauder tree basis and $Y$
be an $X$-singular subspace of $\txtwo$. Then for every
normalized Schauder basic sequence $(x_n)$ in $Y$ there exists
a normalized block sequence $(y_n)$ of $(x_n)$ satisfying
an upper $\ell_2$ estimate. That is, there exists a constant
$C\geq 1$ such that for every $k\in\nn$ and every $a_0,...,a_k\in \rr$
we have
\[ \big\| \sum_{n=0}^k a_n y_n \big\|_{\txtwo} \leq
C \Big( \sum_{n=0}^k |a_n|^2\Big)^{1/2}. \]
In particular, every $X$-singular subspace $Y$ of $\txtwo$
can contain no $\ell_p$ for $1\leq p<2$.
\end{thm}
We notice that in Theorem \ref{t38} one cannot expect to
obtain a block sequence satisfying a lower $\ell_2$ estimate.
Indeed, as it has been shown in \cite[Theorem 25]{AD},
if $\stblng$ is a Schauder tree basis such that the tree
$T$ is not small (precisely, if the tree $T$ contains
a perfect\footnote[4]{A tree $T$ is perfect if every
node $t\in T$ has at least two incomparable successors.}
subtree), then one can find in $\txtwo$ a normalized block
sequence $(x_n)$ which is equivalent to the standard basis of
$c_0$ and which spans an $X$-singular subspace. Clearly, no block
subsequence of $(x_n)$ can have a lower $\ell_2$ estimate.

The rest of this section is organized as follows. In \S 3.2 we
provide a characterization of the class of $X$-singular
subspaces of $\txtwo$. Using this characterization we show,
for instance, that every $X$-singular subspace of $\txtwo$
contains an $X$-compact subspace. This can be seen as a
``tree" version of the classical theorem of T. Kato asserting
that for every strictly singular operator $T:X\to Y$ there
is an infinite-dimensional subspace $Z$ of $X$ such that
the operator $T:Z\to Y$ is compact. In \S 3.3 we give the
proofs of Theorem \ref{t37} and of Theorem \ref{t38}.


\subsection{A characterization of $X$-singular subspaces}
We start with the following definition.
\begin{defn}
\label{d39} Let $\stblng$ be a Schauder tree basis.
The $c_0$ Baire sum of $\mathfrak{X}$, denoted by
$\txco$, is defined to be the completion of $c_{00}(T)$
equipped with the norm
\begin{equation}
\label{e8} \|z\|_{\txco}= \sup\Big\{ \big\| \sum_{t\in \seg}
z(t) x_t\big\|_X : \seg \text{ is a segment of } T \Big\}.
\end{equation}
By $\mathrm{I}:\txtwo\to \txco$ we shall denote the natural inclusion
operator.
\end{defn}
Our characterization of $X$-singular subspaces of $\txtwo$ is
achieved by considering the functional analytic properties of
the inclusion operator $\mathrm{I}:\txtwo\to\txco$. Precisely,
we have the following.
\begin{prop}
\label{p310} Let $\stblng$ be a Schauder tree basis. Let $Y$
be a subspace of $\txtwo$. Then the following are equivalent.
\begin{enumerate}
\item[(i)] $Y$ is an $X$-singular subspace of $\txtwo$.
\item[(ii)] The operator $\mathrm{I}:Y\to\txco$ is strictly singular.
\end{enumerate}
\end{prop}
Let us isolate two consequences of Proposition \ref{p310}. The one
that follows is simply a restatement of Proposition \ref{p310}.
\begin{cor}
\label{c311} Let $\stblng$ be a Schauder tree basis and $Y$ be a
block subspace of $\txtwo$. Assume that $Y$ is $X$-singular. Then
for every $\ee>0$ we may find a finitely supported vector $y\in Y$
with $\|y\|=1$ and such that $\|P_\seg(y)\|\leq \ee$ for every
segment $\seg$ of $T$.
\end{cor}
\begin{cor}
\label{c312} Let $\stblng$ be a Schauder tree basis and $Y$ be an
infinite-dimensional subspace of $\txtwo$. Assume that $Y$ is $X$-singular.
Then there exists an infinite-dimensional subspace $Y'$ of $Y$ which
is $X$-compact.
\end{cor}
\begin{proof}
By Proposition \ref{p310}, the operator $\mathrm{I}:Y\to\txco$
is strictly singular. By \cite[Proposition 2.c.4]{LT}, there
exists an infinite-dimensional subspace $Y'$ of $Y$ such that
the operator $\mathrm{I}:Y'\to \txco$ is compact. It is easy to
see that $Y'$ must be an $X$-compact subspace of $\txtwo$ in
the sense of Definition \ref{d36}(b). The proof is completed.
\end{proof}
For the proof of Proposition \ref{p310} we need a couple
of results from \cite{AD}. The first one is the following
(see \cite[Lemma 17]{AD}).
\begin{lem}
\label{l313} Let $(x_n)$ be a bounded block sequence in $\txtwo$
and $\ee>0$ be such that $\limsup\|P_\sg(x_n)\|<\ee$
for every $\sg\in [T]$. Then there exists $L\in [\nn]$ such that
for every $\sg\in [T]$ we have $|\{n\in L: \|P_\sg(x_n)\|\geq\ee\}|\leq 1$.
\end{lem}
The second result is the following special case of \cite[Proposition 33]{AD}.
\begin{prop}
\label{p314} Let $Y$ be a block $X$-singular subspace of $\txtwo$.
Then for every $\ee>0$ we may find a normalized block sequence
$(y_n)$ in $Y$ such that for every $\sg\in [T]$ we have
$\limsup \|P_\sg(y_n)\|<\ee$.
\end{prop}
We are ready to proceed to the proof of Proposition \ref{p310}.
\begin{proof}[Proof of Proposition \ref{p310}]
It is clear that (ii) implies (i). Hence we only need to show the
converse implication. We argue by contradiction. So, assume that
$Y$ is an $X$-singular subspace of $\txtwo$ such that the operator
$\mathrm{I}:Y\to\txco$ is not strictly singular. By definition,
there exists a further subspace $Y'$ of $Y$ such that $\mathrm{I}:
Y'\to\txco$ is an isomorphic embedding. Using a sliding hump
argument, we may recursively select a normalized basic sequence
$(y_n)$ in $Y'$ and a normalized block sequence $(z_n)$ in
$\txtwo$ such that, setting $Z=\ospan\{z_n:n\in\nn\}$, the following
are satisfied.
\begin{enumerate}
\item[(a)] The sequence $(z_n)$ is equivalent to $(y_n)$.
\item[(b)] The subspace $Z$ of $\txtwo$ is $X$-singular.
\item[(c)] The operator $\mathrm{I}:Z\to\txco$ is an isomorphic
embedding.
\end{enumerate}
The selection is fairly standard (we leave the details
to the interested reader). By (c) above, there exists a constant
$C>0$ such that for every $z\in Z$ we have
\begin{equation}
\label{e9}  C\|z\|_{\txtwo}\leq \|z\|_{\txco} \leq \|z\|_{\txtwo}.
\end{equation}
We fix $k_0\in\nn$ and $\ee>0$ satisfying
\begin{equation}
\label{e10} k_0>\frac{64}{C^4} \ \text{ and } \
\ee<\min\Big\{ \frac{C}{2}, \frac{1}{k_0}\Big\}.
\end{equation}
By (b) above, we may apply Proposition \ref{p314} to the block subspace
$Z$ of $\txtwo$ and the chosen $\ee$. It follows that there exists a
normalized block sequence $(x_n)$ in $Z$ such that $\limsup\|P_\sg(x_n)\|<\ee$
for every $\sg\in [T]$. By Lemma \ref{l313} and by passing to a subsequence
of $(x_n)$ if necessary, we may additionally assume that for every
$\sg\in [T]$ we have $|\{n\in\nn:\|P_\sg(x_n)\|\geq\ee\}|\leq 1$.
As the basis of $\txtwo$ is bi-monotone, we may strengthen this
property to the following one.
\begin{enumerate}
\item[(d)] For every segment $\seg$ of $T$ we have
$|\{n\in\nn: \|P_\seg(x_n)\|\geq\ee\}|\leq 1$.
\end{enumerate}
By Fact \ref{f35} and (\ref{e9}), for every $n\in\nn$ we may
select a segment $\seg_n$ of $T$ such that
\begin{enumerate}
\item[(e)] $\|P_{\seg_n}(x_n)\|\geq C$ and
\item[(f)] $\seg_n\subseteq\{t_k:k\in\range(x_n)\}$.
\end{enumerate}
As the sequence $(x_n)$ is block, we see that such a selection guarantees
that
\begin{enumerate}
\item[(g)] $\|P_{\seg_n}(x_m)\|=0$ for every $n,m\in\nn$ with $n\neq m$.
\end{enumerate}
We set $t_n=\min(\seg_n)$. Applying the classical Ramsey Theorem
we find an infinite subset $L=\{l_0<l_1<l_2<...\}$ of $\nn$ such that
one of the following (mutually exclusive) cases must occur.
\medskip

\noindent \textsc{Case 1.} \textit{The set $\{t_n:n\in L\}$ is an antichain.}
Our hypothesis in this case implies that for every $n,m\in L$ with $n\neq m$ the
segments $\seg_n$ and $\seg_m$ are incomparable.
We define $z=x_{l_0}+...+x_{l_{k_0}}$. As the family
$(\seg_{l_i})_{i=0}^{k_0}$ consists of pairwise incomparable
segments of $T$, we get that
\begin{equation}
\label{e11} \|z\|\geq \Big( \sum_{i=0}^{k_0} \|P_{\seg_{l_i}}(z)\|^2 \Big)^{1/2}
\stackrel{(\mathrm{g})}{=} \Big( \sum_{i=0}^{k_0} \|P_{\seg_{l_i}}(x_{l_i})\|^2 \Big)^{1/2}
\stackrel{(\mathrm{e})}{\geq} C \sqrt{k_0+1}.
\end{equation}
Now we set $w=z/\|z\|\in Z$. Invoking (d) above, inequality
(\ref{e11}) and the choice of $k_0$ and $\ee$ made in (\ref{e10}),
for every segment $\seg$ of $T$ we have
\[ \|P_\seg(w)\|\leq \frac{1+k_0\ee}{C\sqrt{k_0+1}}<\frac{C}{2}. \]
It follows that
\[ \|w\|_{\txco}\leq \frac{C}{2} \]
which contradicts inequality (\ref{e9}). Hence this case is impossible.
\medskip

\noindent \textsc{Case 2.} \textit{The set $\{t_n:n\in L\}$ is a chain.}
Let $\tau\in [T]$ be the branch of $T$ determined by the infinite
chain $\{t_n:n\in L\}$. By (d) above and by passing to an infinite
subset of $L$ if necessary, we may assume that $\|P_\tau(x_n)\|<\ee$
for every $n\in L$. The basis of $\txtwo$ is bi-monotone, and so,
we have the following property.
\begin{enumerate}
\item[(h)] If $\seg$ is a segment of $T$ with $\seg\subseteq \tau$,
then $\|P_\seg(x_n)\|<\ee$ for every $n\in L$.
\end{enumerate}
We set $\seg'_n=\seg_n\setminus \tau$. Observe
that the set $\seg'_n$ is a sub-segment of $\seg_n$.
Notice that $\seg_n$ is the disjoint union of the
successive segments $\seg_n\cap \tau$ and $\seg'_n$.
Hence, by properties (e) and (h) above and the choice
of $\ee$, we see that
\begin{equation}
\label{e12} \|P_{\seg'_n}(x_n)\|\geq C-\ee\geq \frac{C}{2}
\end{equation}
for every $n\in L$. Notice also that if $n,m\in L$ with $n\neq m$,
then the segments $\seg'_n$ and $\seg'_m$ are incomparable.
We set
\[ z=x_{l_0}+...+x_{l_{k_0}} \ \text{ and } \ w=\frac{z}{\|z\|}.\]
Arguing precisely as in Case 1 and using the estimate in (\ref{e12}),
we conclude that
\[ \|w\|_{\txco}\leq \frac{C}{2}.\]
This is again a contradiction. The proof of Proposition \ref{p310}
is completed.
\end{proof}


\subsection{Proof of Theorem \ref{t37} and of Theorem \ref{t38}}
We start with the following lemma.
\begin{lem}
\label{l315} Let $\stblng$ be a Schauder tree basis. Let
$(w_n)$ be a normalized block sequence in $\txtwo$ such
that for every $n\in\nn$ with $n\geq 1$ and every segment
$\seg$ of $T$ we have
\begin{equation}
\label{e14} \|P_\seg(w_n)\|\leq \frac{1}{\sum_{i=0}^{n-1} |\supp(w_i)|^{1/2}}
\cdot \frac{1}{2^{n+2}}.
\end{equation}
Then the following are satisfied.
\begin{enumerate}
\item[(i)] The sequence $(w_n)$ is unconditional.
\item[(ii)] The sequence $(w_n)$ satisfies an upper
$\ell_2$ estimate.
\end{enumerate}
\end{lem}
\begin{proof}
We will only give the proof of part (i). For a proof of part (ii)
we refer to \cite[Proposition 21]{AD}.

So, let $k\in\nn$ and $a_0,...,a_k\in\rr$ be such that
$\|\sum_{n=0}^k a_n w_n\|=1$. Let also $F\subseteq \{0,...,k\}$
with $F=\{n_0<...<n_p\}$ its increasing enumeration. We will
show that $\|\sum_{n\in F} a_n w_n\|\leq \sqrt{3}$. This
will clearly finish the proof. For notational simplicity, we set
\[ w=\sum_{n=0}^k a_n w_n \ \text{ and } \ z=\sum_{n\in F} a_n w_n. \]

Let $(\seg_j)_{j=0}^l$ be an arbitrary collection of pairwise
incomparable segments of $T$. We want to estimate the sum
$\sum_{j=0}^l \|P_{\seg_j}(z)\|^2$. To this end, we may assume
that for every $j\in \{0,...,l\}$ there exists $i\in\{0,...,p\}$
with $\seg_j\cap \supp(w_{n_i})\neq\varnothing$.
We define recursively a partition
$(\Delta_i)_{i=0}^p$ of $\{0,...,l\}$ by the rule
\begin{eqnarray*}
\Delta_0 & = & \big\{ j\in\{0,...,l\}: \seg_j\cap \supp(w_{n_0})\neq\varnothing\big\}\\
\Delta_1 & = & \big\{ j\in\{0,...,l\}\setminus \Delta_0: \seg_j\cap \supp(w_{n_1})\neq
\varnothing\big\}\\
\vdots & & \\
\Delta_p & = & \Big\{ j\in \{0,...,l\}\setminus \Big( \bigcup_{i=0}^{p-1} \Delta_i\Big):
\seg_j\cap \supp(w_{n_p})\neq\varnothing\Big\}.
\end{eqnarray*}
The segments $(\seg_j)_{j=0}^l$ are pairwise incomparable and
a fortiori disjoint. It follows that
\begin{equation}
\label{e15} |\Delta_i|\leq |\supp(w_{n_i})| \ \text{ for every }
i\in \{0,...,p\}.
\end{equation}
Notice also that for every $0\leq i< q\leq p$ we have
\begin{equation}
\label{e16} \sum_{j\in \Delta_q} \|P_{\seg_j}(w_{n_i})\|=0.
\end{equation}

Let $j\in\{0,...,l\}$. There exists a unique $i\in \{0,...,p\}$
such that $j\in \Delta_i$. By Fact \ref{f35}, we may select a
segment $\seg'_j$ of $T$ such that
\begin{enumerate}
\item[(a)] $\seg'_j\subseteq \seg_j$,
\item[(b)] $\seg'_j\subseteq \{t_m:m\in\range(w_{n_i})\}$ and
\item[(c)] $\|P_{\seg_j}(a_{n_i} w_{n_i})\|=
\|P_{\seg'_j}(a_{n_i} w_{n_i})\|$.
\end{enumerate}
The above selection guarantees the following properties.
\begin{enumerate}
\item[(d)] The family $(\seg'_j)_{j=0}^l$ consists of
pairwise incomparable segment of $T$. This is a straightforward
consequence of (a) above and of our assumptions on the family
$(\seg_j)_{j=0}^l$.
\item[(e)] We have $\|P_{\seg_j}(a_{n_i} w_{n_i})\|=
\|P_{\seg'_j}(a_{n_i} w_{n_i})\| = \|P_{\seg'_j}(w)\|$.
This is a consequence of (b) and (c) above and of the fact
that the sequence $(w_n)$ is block.
\end{enumerate}

We are ready for the final part of the argument. Let
$i\in\{0,...,p\}$ and $j\in \Delta_i$. Our goal is
to estimate the quantity $\|P_{\seg_j}(z)\|$. First we notice
that
\begin{eqnarray*}
\|P_{\seg_j}(z)\| & \stackrel{(\ref{e16})}{=} &
\|P_{\seg_j}(a_{n_i} w_{n_i} +...+ a_{n_p} w_{n_p})\| \\
& \leq & \|P_{\seg_j}(a_{n_i}w_{n_i})\|+ \sum_{q=i+1}^p |a_{n_q}|\cdot \|P_{\seg_j}(w_{n_q})\|.
\end{eqnarray*}
Invoking the fact that the Schauder basis $(e_t)_{t\in T}$ of $\txtwo$
is bi-monotone and (\ref{e14}), we see that for every
$q\in\{i+1,...,p\}$ we have $\|P_{\seg_j}(w_{n_q})\|\leq
|\supp(w_{n_i})|^{-1/2} \cdot 2^{-(q+2)}$
and $|a_{n_q}|\leq 1$. Hence, the previous estimate yields
\begin{eqnarray*}
\|P_{\seg_j}(z)\| & \leq &
\|P_{\seg_j}(a_{n_i}w_{n_i})\|+ \frac{1}{|\supp(w_{n_i})|^{1/2}}\cdot
\sum_{q=i+1}^p \frac{1}{2^{q+2}} \nonumber \\
& \stackrel{(\ref{e15})}{\leq} & \|P_{\seg_j}(a_{n_i}w_{n_i})\|+
\frac{1}{|\Delta_i|^{1/2}}\cdot \frac{1}{2^{i+2}} \\
& \stackrel{(\mathrm{e})}{=} & \|P_{\seg'_j}(w)\|+
\frac{1}{|\Delta_i|^{1/2}}\cdot \frac{1}{2^{i+2}}.
\end{eqnarray*}
The above inequality, in turn, implies that if $\Delta_i$ is non-empty,
then
\begin{eqnarray}
\label{e17} \sum_{j\in \Delta_i} \|P_{\seg_j}(z)\|^2 & \leq &
2 \sum_{j\in \Delta_i} \|P_{\seg'_j}(w)\|^2 + 2 \sum_{j\in \Delta_i}
\frac{1}{|\Delta_i|} \cdot \frac{1}{2^{i+2}} \nonumber \\
& \leq & 2 \sum_{j\in \Delta_i} \|P_{\seg'_j}(w)\|^2 + \frac{1}{2^{i+1}}.
\end{eqnarray}
Summarizing, we see that
\[\sum_{j=0}^l \|P_{\seg_j}(z)\|^2  =
\sum_{i=0}^p \sum_{j\in \Delta_i} \|P_{\seg_j}(z)\|^2
\stackrel{(\ref{e17})}{\leq} 2 \sum_{j=0}^l \|P_{\seg'_j}(w)\|^2 + 1
\stackrel{(\mathrm{d})}{\leq} 2 \|w\|^2+1 \leq 3.\]
The family $(\seg_j)_{j=0}^l$ was arbitrary, and so,
$\|z\|\leq \sqrt{3}$. The proof is completed.
\end{proof}
We continue with the proof of Theorem \ref{t37}.
\begin{proof}[Proof of Theorem \ref{t37}]
Let $Y$ be an $X$-singular subspace of $\txtwo$. Clearly
every subspace $Y'$ of $Y$ is also $X$-singular. Hence,
it is enough to show that every $X$-singular subspace
contains an unconditional basic sequence. So, let $Y$
be one. Using a sliding hump argument, we may additionally
assume that $Y$ is a block subspace of $\txtwo$. Recursively
and with the help of Corollary \ref{c311}, we may construct
a normalized block sequence $(w_n)$ in $Y$ such that for
every $n\in\nn$ with $n\geq 1$ and every segment $\seg$ of
$T$ we have
\[ \|P_\seg(w_n)\|\leq \frac{1}{\sum_{i=0}^{n-1} |\supp(w_i)|^{1/2}}
\cdot \frac{1}{2^{n+2}}.\]
By Lemma \ref{l315}(i), the sequence $(w_n)$ is unconditional.
The proof is completed.
\end{proof}
We proceed to the proof of Theorem \ref{t38}.
\begin{proof}[Proof of Theorem \ref{t38}]
Let $Y$ be an $X$-singular subspace of $\txtwo$. Let also $(x_n)$ be
a normalized Schauder basic sequence in $Y$. A standard sliding
hump argument allows us to construct a normalized block sequence
$(v_n)$ of $(x_n)$ and a block sequence $(z_n)$ in $\txtwo$ such that,
setting $Z=\ospan\{z_n:n\in\nn\}$, the following are satisfied.
\begin{enumerate}
\item[(a)] The sequences $(v_n)$ and $(z_n)$ are equivalent.
\item[(b)] The subspace $Z$ of $\txtwo$ is $X$-singular.
\end{enumerate}
As in the proof of Theorem \ref{t37}, using (b) above and
Corollary \ref{c311}, we construct a normalized block sequence
$(w_n)$ of $(z_n)$ such that for every $n\in\nn$ with $n\geq 1$
and every segment $\seg$ of $T$ inequality (\ref{e14}) is satisfied
for the sequence $(w_n)$. By Lemma \ref{l315}(ii), the sequence
$(w_n)$ satisfies an upper $\ell_2$ estimate. Let $(b_n)$ be the
block sequence of $(v_n)$ corresponding to $(w_n)$. Observe that,
by (a) above, the sequence $(b_n)$ is seminormalized and satisfies
an upper $\ell_2$ estimate. The property of being a block sequence
is transitive, and so, $(b_n)$ is a normalized block sequence of
$(x_n)$ as well. Hence, setting $y_n=b_n/\|b_n\|$ for every $n\in\nn$,
we see that the sequence $(y_n)$ is the desired one.

Finally, to see that every $X$-singular subspace of $\txtwo$
can contain no $\ell_p$ for $1\leq p<2$, we argue by contradiction.
So, assume that $Y$ is an $X$-singular subspace of $\txtwo$ containing
an isomorphic copy of $\ell_{p_0}$ for some $1\leq p_0<2$. There exists,
in such a case, a normalized basic sequence $(x_n)$ in $Y$ which is
equivalent to the standard unit vector basis $(e_n)$ of $\ell_{p_0}$.
Let $(y_n)$ be a normalized block subsequence of $(x_n)$ satisfying
an upper $\ell_2$ estimate. As any normalized block subsequence of
$(e_n)$ is equivalent to $(e_n)$ (see \cite{LT}), we see that there
must exist constants $C\geq c>0$ such that for every $k\in\nn$ and every
$a_0,...,a_k\in\rr$ we have
\[ c \Big( \sum_{n=0}^k |a_n|^{p_0}\Big)^{1/p_0} \leq
\big\| \sum_{n=0}^k a_n y_n \big\|_{\txtwo} \leq
C \Big( \sum_{n=0}^k |a_n|^{2}\Big)^{1/2}.\]
This is clearly a contradiction. The proof is completed.
\end{proof}
We close this section by recording the following consequence
of Theorem \ref{t38}.
\begin{cor}
\label{c316} Let $\stblng$ be a Schauder tree basis. Let
$1\leq p<2$. Then the following are equivalent.
\begin{enumerate}
\item[(i)] The space $\txtwo$ contains an isomorphic copy of
$\ell_p$.
\item[(ii)] There exists $\sg\in [T]$ such that $X_\sg$
contains an isomorphic copy of $\ell_p$.
\end{enumerate}
\end{cor}
\begin{proof}
It is clear that (ii) implies (i). Conversely, assume that
$\ell_p$ embeds into $\txtwo$ and let $Y$ be a subspace of $\txtwo$
which is isomorphic to $\ell_p$. By Theorem \ref{t38}, we
see that $Y$ is not $X$-singular. Hence, there exist $\sg\in [T]$
and an infinite-dimensional subspace $Y'$ of $Y$ such that
$P_\sg:Y'\to\xxx_\sg$ is an isomorphic embedding. Recalling
that every subspace of $\ell_p$ contains a copy of $\ell_p$
and that the spaces $\xxx_\sg$ and $X_\sg$ are isometric, the
result follows.
\end{proof}


\section{The main result}

This section is devoted to the proof of Theorem \ref{int1} stated
in the introduction. To this end, we will need the following
correspondence principle between analytic classes of separable
Banach spaces and Schauder tree bases (see \cite[Proposition 83]{AD}
or \cite[Lemma 32]{D}).
\begin{lem}
\label{l41} Let $\aaa'$ be an analytic subset of $\sbs$ such that
every $Y\in\aaa'$ has a Schauder basis. Then there exist a separable
Banach space $X$, a pruned tree $T$ on $\nn\times\nn$ and a normalized
sequence $(x_t)_{t\in T}$ in $X$ such that the following are
satisfied.
\begin{enumerate}
\item[(i)] The family $\stblng$ is a Schauder tree basis.
\item[(ii)] For every $Y\in\aaa'$ there exists $\sg\in [T]$
with $Y\cong X_\sg$.
\item[(iii)] For every $\sg\in [T]$ there exists $Y\in\aaa'$
with $X_\sg\cong Y$.
\end{enumerate}
\end{lem}
We are now ready to proceed to the proof of Theorem \ref{int1}.
So, let $\aaa$ be an analytic subset of $\us$. We apply Proposition
\ref{p21} and we get a subset $\aaa'$ of $\sbs$ with the following
properties.
\begin{enumerate}
\item[(a)] The set $\aaa'$ is analytic.
\item[(b)] Every $Y\in\aaa'$ has a Schauder basis.
\item[(c)] Every $Y\in\aaa'$ is unconditionally saturated.
\item[(d)] For every $X\in\aaa$ there exists $Y\in\aaa'$
such that $Y$ contains an isometric copy of $X$.
\end{enumerate}
By (a) and (b) above, we apply Lemma \ref{l41} to the set
$\aaa'$ and we get a Schauder tree basis $\stblng$ satisfying
the following.
\begin{enumerate}
\item[(e)] For every $Y\in\aaa'$ there exists $\sg\in [T]$ with
$Y\cong X_\sg$.
\item[(f)] For every $\sg\in [T]$ there exists $Y\in\aaa'$
such that $X_\sg\cong Y$.
\end{enumerate}
Consider the $\ell_2$ Baire sum $\txtwo$ of this Schauder tree
basis $\mathfrak{X}$. We claim that the space $\txtwo$ is the
desired one. Indeed, recall first that $\txtwo$ has a Schauder basis.
Moreover, by  (d) and (e) above we see that $\txtwo$ contains an
isomorphic copy of every $X\in\aaa$.

What remains is to check that $\txtwo$ is unconditionally saturated.
To this end, let $Z$ be an arbitrary subspace of $\txtwo$. We have to
show that the space $Z$ contains an unconditional basic sequence. We
distinguish the following (mutually exclusive) cases.
\medskip

\noindent \textsc{Case 1.} \textit{The subspace $Z$ is not $X$-singular.}
In this case, by definition, there exist $\sg\in [T]$ and a further
subspace $Z'$ of $Z$ such that the operator $P_\sg:Z'\to \xxx_\sg$ is
an isomorphic embedding. By (f) and (c) above, we get that $Z'$
must contain an unconditional basic sequence.
\medskip

\noindent \textsc{Case 2.} \textit{The subspace $Z$ is $X$-singular.}
By Theorem \ref{t37}, we see that in this case the subspace $Z$ must
also contain an unconditional basic sequence.
\medskip

\noindent By the above, it follows that $\txtwo$ is unconditionally
saturated. The proof of Theorem \ref{int1} is completed.


\end{document}